\documentclass[letterpaper, 10 pt, conference]{ieeeconf}

\IEEEoverridecommandlockouts                              % This command is only needed if you want to use the \thanks command
\makeatletter
\let\NAT@parse\undefined
\makeatother

\usepackage{amsmath}
\usepackage{amssymb}
\usepackage{mathtools}
\usepackage{mathrsfs}
\usepackage{comment}

\newcommand{\R}{\mathbb{R}}		% real numbers
\newcommand{\Rn}{\mathbb{R}^n}		% n-vectors with real elements
\newcommand{\M}{\mathcal{M}}		% finite signed measures
\renewcommand{\P}{\mathcal{P}}		% probability measures
\newcommand{\E}{\mathbb{E}}		% expectation
\newcommand{\N}{\mathcal{N}}		% normal distribution
\DeclareMathOperator*{\supp}{supp}	% support
\DeclareMathOperator*{\NE}{NE}		% nash equilibria
\newcommand{\TV}{\textup{TV}}		% total variation
\DeclareMathOperator*{\tr}{tr}			% trace

\newcommand{\inlinesubsection}[1]{%
    \vspace{0.5em}% Adjust the vertical space as needed
    \noindent\textbf{#1.}% Bold title with a period
    \hspace{0.5em}% Space after the title before the text
}

\usepackage{amsthm}
\theoremstyle{plain}
\newtheorem{theorem}{Theorem}
\newtheorem{proposition}{Proposition}

\theoremstyle{definition}
\newtheorem{definition}{Definition}
\newtheorem{example}{Example}

\theoremstyle{remark}
\newtheorem{remark}{Remark}

\makeatletter
\let\NAT@parse\undefined
\makeatother
\usepackage[numbers]{natbib}
\renewcommand*\citealt[2][]{\citet[#1]{#2}}

\usepackage{subcaption}			% for subfigures
\usepackage{booktabs,multirow}		% for tables
\usepackage{pgf}
\usepackage{graphicx,pgfplots,tikzscale}	% for matlab2tikz
\usetikzlibrary{calc}			% for coordinate computations
\pgfplotsset{compat=newest}
\pgfplotsset{plot coordinates/math parser=false}

\usepackage{hyperref}		% custom hyperlinks
\hypersetup{%
	colorlinks=true,%
	linkcolor=black,%
	urlcolor=black,%
	citecolor=black,%
	filecolor=black%
}				% hyperlink style settings
\usepackage[all]{hypcap}	% hyperlink to figure, not caption

\title{\LARGE \bf
	Approximately Gaussian Replicator Flows:\\
	Nonconvex Optimization as a Nash-Convergent Evolutionary Game
}

\author{%
	Brendon G.\ Anderson\authorrefmark{1} \and Samuel Pfrommer\authorrefmark{2} \and Somayeh Sojoudi\authorrefmark{2}%
	\thanks{\authorrefmark{1}Department of Mechanical Engineering, California Polytechnic State University. Email: {\tt \href{mailto:bga@calpoly.edu}{bga@calpoly.edu}}.}
	\thanks{\authorrefmark{2}Department of Electrical Engineering and Computer Sciences, University of California, Berkeley. Emails: {\tt \{\href{mailto:sam.pfrommer@berkeley.edu}{sam.pfrommer}, \href{mailto:sojoudi@berkeley.edu}{sojoudi}\}@berkeley.edu}.}%
	\thanks{This work was supported by grants from ONR, NSF, and ARO, and was conducted while Brendon G.\ Anderson was with the University of California, Berkeley.}%
}

\begin{document}

\maketitle
\thispagestyle{empty}
\pagestyle{empty}

\begin{abstract}
	This work leverages tools from evolutionary game theory to solve unconstrained nonconvex optimization problems. Specifically, we lift such a problem to an optimization over probability measures, whose minimizers exactly correspond to the Nash equilibria of a particular population game. To algorithmically solve for such Nash equilibria, we introduce \emph{approximately Gaussian replicator flows} (AGRFs) as a tractable alternative to simulating the corresponding infinite-dimensional replicator dynamics. Our proposed AGRF dynamics can be integrated using off-the-shelf ODE solvers when considering objectives with closed-form integrals against a Gaussian measure. We theoretically analyze AGRF dynamics by explicitly characterizing their trajectories and stability on quadratic objective functions, in addition to analyzing their descent properties. Our methods are supported by illustrative experiments on a range of canonical nonconvex optimization benchmark functions.
\end{abstract}

\section{Introduction}
\label{sec: intro}

We consider solving unconstrained optimization problems of the form
\begin{equation*}
	\inf_{x\in \Rn} f(x),
\end{equation*}
where $f$ is generally nonconvex. Such optimization problems arise naturally across a range of scientific and engineering disciplines, including computational chemistry \cite{chan2019bayesian}, computer vision \cite{ma2018accelerated}, and medicine \cite{gramfort2012mixed}. Notably, the classical empirical risk minimization approach in machine learning amounts to an unconstrained nonconvex optimization problem \cite{jain2017non}.

Evolutionary algorithms comprise one popular family of approaches for solving nonconvex problems. These algorithms stochastically search the parameter space by evolving a population of candidate solutions according to biologically-inspired variation and selection operations \cite{sloss20202019}. Unlike gradient-based methods like stochastic gradient descent, practical evolutionary algorithms are generally intractable for theoretical analysis due to the handcrafted nature of their update rules. Existing analyses are thus typically limited to a very simple $(1+1)$-evolution strategy \cite{glasmachers2020global, rudolph2013convergence}, or a discrete optimization setting \cite{he2015average}. Conversely, negative results show that evolutionary algorithms cannot converge exponentially fast for certain classes of objective functions \cite{tarlowski2023asymptotic}.

Instead of relying on handcrafted evolution heuristics, we propose an evolutionarily-inspired optimization method that naturally arises from a principled recasting of the optimization problem as a population game. In doing so, we achieve the following contributions:
\begin{enumerate}
	\item We develop a framework for reformulating nonconvex optimization as a population game, which we show can be solved to global optimality by simulating the replicator dynamics for continuous objectives with unique global minima.
	\item We propose \emph{approximately Gaussian replicator flows} (AGRFs) as a tractable alternative to simulating the replicator dynamics, and theoretically analyze the resulting ordinary differential equation's descent properties, in addition to determining its explicit solutions for quadratic objective functions.
	% \item We bound the impact of constraining the Gaussian covariance matrix to be diagonal for high-dimensional inputs.
	\item We provide illustrative experiments using canonical nonconvex optimization test functions, demonstrating that our optimization method can escape local minima, adapt to the optimization landscape, and recover from poorly conditioned initializations.
\end{enumerate}

\subsection{Related Works}
Our work, drawing on the frameworks and tools of population games and evolutionary dynamics \citep{sandholm2010population}, is intimately related to evolutionary algorithms for optimization. Some of the most classical of such methods are evolution strategies (ESs), wherein a population of candidate solutions iteratively evolves according to stochastic mutation and selection rules \citep{rechenberg1973evolutionsstrategie}. Covariance matrix adaptation evolution strategy (CMA-ES) extends ESs by allowing the covariance matrix of the ES sampling distribution to adapt to the geometry of the objective function \citep{hansen2001completely}. CMA-ES is often considered the state-of-the-art in evolutionary computation \citep{glasmachers2020global}, yet it has also been shown to be subsumed by natural evolution strategies (NESs) \citep{akimoto2012theoretical}.

Unlike the heuristic approach to ES-based update rules, NES provides a principled optimization method by using natural gradient descent \citep{amari1998natural} to minimize the expected objective loss over a statistical manifold of parameterized probability distributions \citep{wierstra2014natural}. Closely related are estimation of distribution algorithms (EDAs) \citep{larranaga2001estimation}, which explicitly build, sample from, and update probability distributions to represent promising candidate solutions to the optimization, albeit these distributions are typically constructed via maximum likelihood estimation. Generalizing ES, CMA-ES, and NES, the information geometric optimization (IGO) framework poses optimization as a continuous time natural gradient flow on a statistical manifold \citep{ollivier2017information}. IGO algorithms are constructed by discretizing both time (via Euler's method) and spatial integration (via sampling) to approximate the flow.

All of the aforementioned evolutionary algorithms view the objective $f$ as a black box, and as such, rely on evaluating $f$ on a finite number of random samples. In practice, the challenges posed by the objective function $f$ are often structural and not due to a lack of information, e.g., training machine learning models often relies on nonconvex and nonsmooth objectives \citep{jain2017non,fattahi2020exact,anderson2019global}, whose functional forms are known. Our framework is capable of leveraging closed-form integrals of nonconvex yet white-box objectives $f$ against a Gaussian distribution, in order to pose evolutionary optimization as a \emph{deterministic} differential equation. That is, unlike the above stochastic search methods, which often rely on heuristically designed algorithmic rules and parameters, our evolutionarily-inspired optimization framework can be methodically implemented via off-the-shelf ordinary differential equation (ODE) solvers without random sampling.

Two of the works closest related to ours are \citet{jacimovic2022natural,jacimovic2023fundamental}. The authors draw links between the replicator dynamics and minimization of a bilinear-quadratic objective function by showing that replicator trajectories are natural gradient flows for such objectives. This characterization relies on inverting the Fisher information matrix to evaluate the flow's vector field, whereas our approximately Gaussian replicator flows admit closed-form solutions for general quadratic functions. Our framework is particularly distinct due to the fact that our proposed AGRFs allow for using off-the-shelf ODE solvers whenever the objective is closed-form integrable against a Gaussian measure, e.g., general polynomials and sinusoids.

\subsection{Notations}
\label{sec: notations}

%The Frobenius norm of a matrix $A\in\R^{m\times n}$ is denoted by $\|A\|_F$.
We denote the identity matrix by $I$, the dimension of which will always be clear from context. If $X$ is a topological space, the set of finite signed Borel measures on $X$ is denoted by $\M(X)$, and the set of Borel probability measures is denoted by $\P(X)$. The support of a measure $\mu \in \M(X)$ is denoted by $\supp(\mu)$. For $\mu\in\M(X)$ and a measurable real-valued function $f$ on $X$, we define $\left<f,\mu\right> \coloneqq \int_X f d\mu$, which we allow to be infinite.
% The supremum norm of a bounded function $f$ on $X$ is denoted by $\|f\|_\infty = \sup_{x\in X}|f(x)|$.
The derivative of a measure-valued mapping $\mu\colon [0,\infty)\to\M(X)$ at $t\in[0,\infty)$, if it exists, is the measure $\dot{\mu}(t)\in\M(X)$ defined by 
\[
	\lim_{\epsilon \to 0} \left\|\tfrac{\mu(t+\epsilon) - \mu(t)}{\epsilon} - \dot{\mu}(t)\right\|_\TV = 0,
\]
where $\|\cdot\|_\TV$ is the total variation norm on $\M(X)$. The Dirac distribution at $x\in X$ is denoted by $\delta_x$, and the Gaussian distribution on $\Rn$ with mean $m\in\Rn$ and covariance $C\in\R^{n\times n}$ is denoted by $\N(m,C)$.

\section{Optimization as an Infinite-Strategy Game}
\label{sec: method}

We seek to solve the unconstrained optimization problem
\begin{equation}
	p^\star \coloneqq \inf_{x\in X} f(x),
	\label{eq: pstar}
\end{equation}
over the Euclidean space $X = \Rn$ with a possibly nonsmooth and nonconvex objective function $f\colon X \to \R$. We assume that $f$ is Borel measurable and attains a minimum.

The problem $p^\star$ is equivalent to an infinite-dimensional convex optimization problem over probability measures \citep{anderson2023tight}:
\begin{equation}
	p^\star = p' \coloneqq \inf_{\mu\in\P(X)} \left<f,\mu\right>.
	\label{eq: pprime}
\end{equation}
It is easy to see that $p'$ also attains a minimizer, and that, if $\mu^\star$ solves $p'$, then $x^\star$ solves $p^\star$ for all $x^\star \in \supp(\mu^\star)$. Thus, our attention may be refocused on solving $p'$. We now show that solving $p'$ amounts to finding a Nash equilibrium of a certain population game.

\subsection{Turning Optimizers into Nash Equilibria}

We briefly recall the key notions used in the study of population games and evolutionary dynamics---see \citet{sandholm2010population} for a more thorough treatment.

Consider a large population of strategic individuals playing a game. The players choose strategies from a \emph{strategy set} $S$, which is assumed to be a topological space. A distribution of strategy choices employed across the population is encoded by a probability measure $\mu \in \P(S)$, termed a \emph{population state}. Every population state has an associated \emph{mean payoff function} $F_\mu \colon S \to \R$, assumed to be continuous, with $F_\mu(s)$ giving the payoff to strategy $s$ when the population state is $\mu$. The mapping $F\colon \mu \mapsto F_\mu$ is called the \emph{population game}. A \emph{Nash equilibrium} of the game $F$ is a population state $\mu$ such that $\left<F_\mu,\nu\right> \le \left<F_\mu,\mu\right>$ for all $\nu\in\P(S)$, and, if the inequality holds strictly for all $\nu \ne \mu$, then $\mu$ is called a \emph{strict Nash equilibrium} of $F$. The set of all Nash equilibria of $F$ is denoted by $\NE(F)$.

Throughout the remainder of the paper, we consider a population game where the strategy set is $X$, and the mean payoff functions are given by \emph{excess payoffs} under the objective function $f$:
\begin{equation*}
	F_\mu(x) = \left<f,\mu\right> - f(x).
\end{equation*}
Such mean payoff functions can be written as the average of the \emph{relative payoffs} $g(x,x') \coloneqq f(x')-f(x)$ with respect to the population state $\mu$:
\begin{equation*}
	F_\mu(x) = \int_X g(x,x') d\mu(x').
\end{equation*}
Note that this is an infinite-strategy population game since $X=\Rn$ is continuous. A key feature of this game is that it puts optimizers of $p'$ in one-to-one correspondence with Nash equilibria:
\begin{proposition}
	\label{prop: nash_are_global_optima}
	A distribution $\mu\in\P(X)$ is a Nash equilibrium of $F$ if and only if $\mu$ is a global minimizer of $p'$. Furthermore, $\mu$ is a strict Nash equilibrium of $F$ if and only if $\mu$ is the unique global minimizer of $p'$.
\end{proposition}
\begin{proof}
	Let $\mu\in\NE(F)$. Then, by definition of Nash equilibria, it holds for all $\nu\in\P(X)$ that
	\begin{align*}
		&\int_X \int_X (f(x') - f(x)) d\mu(x') d\nu(x) \\
		& \quad \le \int_X \int_X (f(x') - f(x)) d\mu(x') d\mu(x),
	\end{align*}
	and hence
	\begin{equation*}
		\left<f,\mu\right> - \left<f,\nu\right> \le \left<f,\mu\right> - \left<f,\mu\right> = 0,
	\end{equation*}
	proving that $\left<f,\mu\right> \le \left<f,\nu\right>$. Thus, $\mu$ solves $p'$.

	On the other hand, suppose that $\mu\in\P(X)$ solves $p'$, so that $\left<f,\mu\right> \le \left<f,\nu\right>$ for all $\nu\in\P(X)$. Then reversing the above line of analysis yields that $\left<F(\mu),\nu\right> \le \left<F(\mu),\mu\right>$ for all $\nu\in\P(X)$, implying that $\mu\in\NE(F)$. The second assertion follows by the same reasoning using the appropriate strict inequalities.
\end{proof}

Another characteristic of this game is that it is \emph{monotone}, meaning that $\left<F_\mu-F_\nu,\mu-\nu\right> \le 0$ for all $\mu,\nu\in\P(X)$:
\begin{proposition}
	\label{prop: monotone}
	The population game $F$ is monotone.
\end{proposition}
\begin{proof}
	Let $\mu,\nu \in \P(X)$. Then
	\begin{align*}
		&\left<F_\mu-F_\nu,\mu-\nu\right> \\
		& \quad = \int_X \int_X (f(x')-f(x)) d(\mu-\nu)(x') d(\mu-\nu)(x) \\
		& \quad = 0.
	\end{align*}
\end{proof}
Monotonicity of population games has been found to be a useful condition for ensuring the convergence of evolutionary dynamics to Nash equilibria---see, e.g., \cite{hofbauer2009stable,fox2013population,park2019population,arcak2021dissipativity,anderson2023evolutionary}.

\subsection{Replicator Dynamics as an Optimizing Process}
The \emph{replicator dynamics} are one of the most popular evolutionary dynamics models (EDMs). EDMs are processes in which players continuously revise their strategies according to a certain revision protocol \citep{sandholm2010population}. Specifically, the replicator dynamics model the situation in which players choose to imitate the strategy of a random opponent who currently has higher payoff, with probability proportional to the current difference in their payoffs. Mathematically, the replicator dynamics read
\begin{equation}
	\dot{\mu}(t)(B) = \int_B \left(E(\delta_x,\mu(t)) - E(\mu(t),\mu(t))\right)d(\mu(t))(x)
	\label{eq: replicator}
\end{equation}
for all Borel sets $B$ and times $t\in[0,\infty)$, where $E \colon \P(X) \times \P(X) \to \R$, given by
\begin{equation*}
	E(\nu,\mu) = \int_X F_\mu d\nu = \left<f,\mu\right> - \left<f,\nu\right>,
\end{equation*}
encodes the \emph{average mean payoff} to a population state $\nu\in\P(X)$ relative to a population state $\mu\in\P(X)$ according to the payoffs of our game $F$. Our key insight is the following: if the replicator dynamics converge to a Nash equilibrium of our game, then simulating them will lead us to a global minimizer of $p'$ according to \autoref{prop: nash_are_global_optima}. Indeed, the following result shows that the monotonicity of our game ensures that this convergence happens in the case that $f$ is continuous and has a unique global minimizer, even when $f$ is nonconvex and nonsmooth.
\begin{theorem}
	\label{prop: replicator_converges}
	Let $\mu_0\in\P(X)$, and let $\mu\colon[0,\infty) \to \P(X)$ be a solution to the replicator dynamics with initial condition $\mu(0) = \mu_0$. If $f$ is continuous and has a unique global minimizer $x^\star \in X$, then it holds that $\mu(t)$ converges weakly to $\mu^\star \coloneqq \delta_{x^\star}$ as $t\to\infty$.
\end{theorem}
\begin{proof}
	Suppose that $f$ has a unique global minimizer $x^\star \in X$. Clearly, $\mu^\star = \delta_{x^\star}$ is the unique global minimizer of $p'$. Thus, $\mu^\star$ is a strict Nash equilibrium of $F$ by \autoref{prop: nash_are_global_optima} and hence $\left<F_{\mu^\star},\nu\right> < \left<F_{\mu^\star},\mu^\star\right>$ for all $\nu \in \P(X) \setminus \{\mu^\star\}$. By \autoref{prop: monotone}, it holds for all such $\nu$ that
	\begin{align*}
		& \left<F_\nu,\nu\right> - \left<F_\nu,\mu^\star\right> \\
		&\quad = \left<F_{\mu^\star},\nu\right> - \left<F_{\mu^\star},\mu^\star\right> + \left<F_\nu - F_{\mu^\star},\nu-\mu^\star\right> \\
		&\quad < 0.
	\end{align*}
	Thus, $\mu^\star$ is strongly uninvadable with uniform invasion barrier $\epsilon = \infty$, per \citet[Definition~5]{oechssler2001evolutionary}. The result then follows from \citet[Theorem~3]{oechssler2001evolutionary}.
\end{proof}

\subsection{Simulating the Replicator Dynamics} \label{sec: simulating}

\autoref{prop: replicator_converges} shows that we can find a globally optimal solution to \eqref{eq: pstar} by simulating the replicator dynamics \eqref{eq: replicator} with our constructed population game $F$. However, this constitutes solving an infinite-dimensional differential equation in general. It was shown in \citet{cressman2006stability} that, if $\mu\colon [0,\infty) \to \P(X)$ is a solution to the replicator dynamics, then its mean and covariance
\begin{align*}
	m(t) &\coloneqq \int_X x d(\mu(t))(x), \\
	C(t) &\coloneqq \int_X x x^\top d(\mu(t))(x) - m(t) m(t)^\top,
\end{align*}
evolve according to the finite-dimensional ordinary differential equations
\begin{align*}
	\dot{m}_i(t) &= -\left. \frac{\partial}{\partial \lambda_i} E(\mu_\lambda (t), \mu(t)) \right|_{\lambda = 0}, \\
	\dot{C}_{ij}(t)  &= \left. \frac{\partial^2}{\partial \lambda_i \partial \lambda_j} E(\mu_\lambda(t), \mu(t)) \right|_{\lambda = 0},
\end{align*}
for all $i,j\in\{1,\dots,n\}$, where $\mu_\lambda(t) \in \P(X)$ is the measure defined by
\begin{equation*}
	\mu_\lambda(t)(B) = \frac{\int_B \exp(-\lambda^\top x)d(\mu(t))(x)}{\int_X \exp(-\lambda^\top x) d(\mu(t))(x)}.
\end{equation*}
It is easy to show that these ODEs can be rewritten as
\begin{equation}
	\begin{aligned}
		\dot{m}_i(t) ={}&m_i(t) \E_{x\sim \mu(t)}[f(x)] - \E_{x\sim \mu(t)}[x_i f(x)], \\
		\dot{C}_{ij}(t) ={}&(C_{ij}(t) - m_i(t) m_j(t))\E_{x\sim \mu(t)}[f(x)] \\
		& \quad - \E_{x\sim \mu(t)} [x_i x_j f(x)] \\
				&\quad + m_i(t) \E_{x\sim \mu(t)}[x_j f(x)] \\
				&\quad + m_j(t) \E_{x\sim \mu(t)}[x_i f(x)].
	\end{aligned}
	\label{eq: odes}
\end{equation}
\citet{cressman2006stability} showed that, in the case that the relative payoff function $g$ takes a bilinear-quadratic form, the set of Gaussian distributions is invariant under \eqref{eq: replicator}, and therefore these two finite-dimensional ODEs fully characterize the replicator dynamics. In such cases, it holds that $C(t) \to 0$ as $t\to\infty$, and therefore the trajectory $t\mapsto \mu(t)$ converges towards $t\mapsto \delta_{m(t)}$ (in the weak topology). Thus, simulating the above finite-dimensional ODEs to generate a limiting Dirac measure $\delta_{m^\star}$ naturally gives us a minimizer $m^\star$ for this bilinear-quadratic case, through an application of \autoref{prop: replicator_converges}.

In the general case, solving the ODEs \eqref{eq: odes} is intractable, as computing the expectations cannot be carried out for an arbitrary distribution $\mu(t)$ and objective function $f$. To make the differential equations solvable (using off-the-shelf ODE solvers), we propose to approximate $\mu(t)$ by $\N(m(t),C(t))$ at all times $t$. This leads to our primary definition:

\begin{definition}
	A solution $(m,C) \colon [0,\infty) \to \Rn \times \R^{n\times n}$ to the system of ordinary differential equations
	\begin{equation}
		\begin{aligned}
			\dot{m}_i(t) ={}&m_i(t) \E_{x\sim \N(m(t),C(t))}[f(x)] \\
			& \quad - \E_{x\sim \N(m(t),C(t))}[x_i f(x)], \\
			\dot{C}_{ij}(t) ={}&(C_{ij}(t) - m_i(t) m_j(t))\E_{x\sim \N(m(t),C(t))}[f(x)] \\
			& \quad - \E_{x\sim \N(m(t),C(t))} [x_i x_j f(x)] \\
					&\quad + m_i(t) \E_{x\sim \N(m(t),C(t))}[x_j f(x)] \\
					&\quad + m_j(t) \E_{x\sim \N(m(t),C(t))}[x_i f(x)], \\
		\end{aligned}
		\label{eq: agrf}
	\end{equation}
	is called an \emph{approximately Gaussian replicator flow (AGRF)} under $f$.
\end{definition}

The right-hand sides of \eqref{eq: agrf} can be determined explicitly for any objective functions $f$ for which we can compute the appropriate Gaussian expectations in closed form. This class of objective functions, which we refer to as \emph{Gaussian-integrable}, includes a wide variety of functions encountered in practice, as we will now discuss further.

\subsection{Closed-Form AGRF Dynamics}
\label{eq: closed-form_agrf_dynamics}

Explicitly writing out the AGRF dynamics \eqref{eq: agrf} requires computing the expectations of $f(x)$, $x_i f(x)$, and $x_i x_j f(x)$ under Gaussian-distributed vectors $x$. For general objectives $f$, these quantities can be approximated using Monte Carlo schemes. However, evaluating the desired estimates to sufficient resolution with such an approach is computationally intensive. We thus restrict our focus to Gaussian-integrable objective functions throughout the remainder of the paper, with polynomials and sinusoids serving as running examples.

\subsubsection{Polynomials}
For polynomials, the Gaussian expectations can be symbolically derived in closed-form from the following well-known identity \citep{barvinok2007integration}:
\begin{equation*}
    \E_{x \sim \N(0, I)} [x^{\alpha}] = 
    \begin{cases}
        \pi^{-\frac{n}{2}} \prod_{i=1}^n 2^{\frac{\alpha_i}{2}} \Gamma\left(\frac{\alpha_i + 1}{2} \right) & \text{$\alpha_i$ all even}, \\
        0 & \text{otherwise,}
    \end{cases}
\end{equation*}
where $x^{\alpha} = x_1^{\alpha_1} \cdots x_n^{\alpha_n}$ is a monomial of order $\alpha = (\alpha_1,\dots,\alpha_n)$ and $\Gamma$ is Euler's gamma function. This result is generalized to arbitrary mean and covariance by symbolically applying an affine transformation to $x$, and substituting for a particular mean $m$ and covariance $C$. The extension from monomials to polynomials follows trivially from linearity of integration. Notably, since $x_i f(x)$ and $x_i x_j f(x)$ are all polynomials if $f(x)$ is a polynomial, we can integrate all these expressions using the same approach.

\subsubsection{Sinusoids}
For sinusoids (including cosine), we can derive closed-form expectations for complex exponentials $x\mapsto e^{i a^\top x}$ by computing the Fourier transform of multivariate Gaussian density functions. Straightforward calculus results in the following expressions, where, here, we use $j,k$ as indices to reserve the symbol $i$ for the imaginary unit:
\begin{align*}
	\E_{x \sim \N(m, C)} \left[e^{i a^\top x}\right] &= e^{i a^\top m - \frac{1}{2}a^\top C a} \eqqcolon \overline{m}, \\
	\E_{x \sim \N(m, C)} \left[x_j e^{i a^\top x}\right] &= \left(m_j + i C_j^\top a\right) \overline{m}, \\
	\E_{x \sim \N(m, C)} \left[x_j x_k e^{i a^\top x}\right] &= \Big(C_{jk} + m_j m_k - (C_j^\top a)(C_k^\top a) \\
	&\qquad + i(m_j C_k^\top a + m_k C_j^\top a) \Big) \overline{m},
\end{align*}
where $C_j$ is the $j$th row of $C$. The appropriate expectations for sinusoids are then derived using Euler's formula.

\begin{remark}
Our proposed approach for deterministically approximating the replicator dynamics using AGRFs intimately relies on the closed-form computation of the ODE derivatives, which requires Gaussian-integrable objective functions. We highlight two potential extensions to more complex objectives which are not as well-behaved. The first consists of estimating expectations using a Monte Carlo scheme and leveraging ODE integration methods which are better suited for noisy derivatives. The second avenue involves using existing approximation techniques to locally represent the objective in a Gaussian-integrable form, e.g., Taylor expansions for polynomials, or Fourier series for sinusoids. Given such a local approximation, the ODE derivatives are again evaluable in closed-form. Investigating these extensions is an exciting area of future work.
\end{remark}

\section{Theoretical Analysis of AGRFs}
\label{sec: theory}

We now theoretically analyze our proposed approximately Gaussian replicator flows. Our analysis emphasizes AGRFs for quadratic objective functions, which serves to provide insight into the AGRF behavior in the vacinity of local minima of more general twice continuously differentiable functions.
% \autoref{sec: diagonal} explores the impact of restricting the covariance matrix to be diagonal as a means to increase numerical efficiency.
We begin by explicitly characterizing the governing equations of the AGRF for a quadratic objective function.

\subsection{Explicit Characterization of Quadratic Flows}

The following theorem is similar to \citet[Theorem~1]{cressman2006stability}, albeit their result is restricted to evolutionary games with bilinear-quadratic payoffs of the form $g(x,x') = -x^\top x + x^\top Q x'$, and thus cannot directly be used for our purposes where $g(x,x') = f(x')-f(x)$.

\begin{theorem}
	\label{prop: quadratic_flow}
	Consider $f\colon \Rn \to \R$ given by $f(x) = x^\top A x + b^\top x + c$ with $A\in\R^{n\times n}$ symmetric. The AGRF under $f$ is given by
	\begin{equation}
		\begin{aligned}
			\dot{m}(t) &= -2 C(t) A m(t) - C(t) b\\
			\dot{C}(t) &= -2 C(t) A C(t).
		\end{aligned}
		\label{eq: quadratic_flow}
	\end{equation}
\end{theorem}

\begin{proof}
	Let $i,j\in\{1,\dots,n\}$. Using the formulas for Gaussian expectations of linear, quadratic, and cubic forms from \citet{petersen2008matrix} and \citet{jacimovic2023fundamental}, we have that
	% Notice that the formula for cubic expectation in Jacimovic, 2023, does not match matrix cookbook (the quadratic form matrix and the covariance are flipped in one of the terms). I am sticking with the matrix cookbook, since it seems to yield results coincident with Cressman, 2006, for quadratic form costs.
	\begin{align*}
		\dot{m}_i(t) &= m_i(t) \E_{x\sim\N(m(t),C(t))}[x^\top A x + b^\top x + c] \\
		& \qquad - \E_{x\sim\N(m(t),C(t))}[x_i (x^\top Ax + b^\top x + c)] \\
		&= m_i(t) \left(\tr(AC(t)) + m(t)^\top A m(t) + b^\top m(t) + c\right) \\
		&\qquad - m_i(t) m(t)^\top A m(t) - 2e_i^\top C(t) A m(t) \\
		&\qquad - m_i(t) \tr(AC(t)) - m_i(t) b^\top m(t) \\
		&\qquad - e_i^\top C(t) b - m_i(t) c \\
		&= -2e_i^\top C(t) A m(t) - e_i^\top C(t) b.
	\end{align*}
	We repeat a similar computation for the covariance matrix. For conciseness, we drop the time argument $t$ for both $C$ and $m$:	
	\begin{align*}
		\dot{C}_{ij} &= \left(C_{ij} - m_im_j\right) \E_{x\sim\N(m,C)}[x^\top A x + b^\top x + c] \\
		&\qquad - \E_{x\sim\N(m,C)}[x_i x_j (x^\top Ax + b^\top x + c)] \\
		&\qquad + m_i\E_{x\sim\N(m,C)}[x_j (x^\top Ax + b^\top x + c)] \\
		&\qquad + m_j \E_{x\sim\N(m,C)} [x_i (x^\top A x + b^\top x + c)] \\
		&= \left(C_{ij} - m_i m_j\right) \left(\tr(AC) + m^\top A m + b^\top m + c\right) \\
		&\qquad - 2 e_i^\top C A C e_j - 2 m^\top AC(m_j e_i + m_i e_j) \\
		&\qquad - \left(\tr(AC)+m^\top Am\right)\left(C_{ij}+m_im_j\right) \\
		&\qquad - m_im_j b^\top m - \left(m_i e_j^\top + m_j e_i^\top\right)C b - b^\top m C_{ij} \\
		&\qquad - c\left(m_im_j + C_{ij}\right) \\
		&\qquad + m_i\Big(m_j m^\top A m + 2e_j^\top C A m + m_j \tr(AC) \\
		&\quad\qquad + m_j b^\top m + e_j^\top C b + m_j c\Big) \\
		&\qquad + m_j\Big(m_i m^\top A m + 2e_i^\top C A m + m_i \tr(AC) \\
		&\quad\qquad + m_i b^\top m + e_i^\top C b + m_i c\Big).
	\end{align*}
	After simplification, we conclude that
	\begin{equation*}
		\dot{C}_{ij}(t) = -2e_i^\top C(t)AC(t)e_j.
	\end{equation*}
\end{proof}

The ODE governing $C$ in \eqref{eq: quadratic_flow} is autonomous, and independent of $m$. Its solution is known in closed-form:

\begin{proposition}[\citealt{jacimovic2022natural}]
	\label{prop: quadratic_flow_cov}
	Consider $f\colon\Rn\to \R$ given by $f(x) = x^\top A x + b^\top x + c$ with $A\in\mathbb{R}^{n\times n}$ symmetric. Then $C\colon [0,\infty) \to \R^{n\times n}$ given by
	\begin{equation*}
		C(t) = \left(C(0)^{-1} + 2t A\right)^{-1},
	\end{equation*}
	if it exists for all $t\in[0,\infty)$, solves the ODE $\dot{C}(t) = -2C(t)AC(t)$ of the AGRF under $f$ given by \eqref{eq: quadratic_flow}.
\end{proposition}

The inverse matrix $(C(0)^{-1}+2tA)^{-1}$ exists for all $t$ whenever $C(0)$ is positive definite and $A$ is positive semidefinite.

Next, we identify a closed-form solution to the remaining part of the AGRF for quadratic $f$, namely, the solution of the ODE governing $m$. Such a solution was not previously determined in \citet{cressman2006stability} or \citet{jacimovic2022natural}.

\begin{proposition}
	\label{prop: quadratic_flow_mean}
	Consider $f\colon\Rn\to \R$ given by $f(x) = x^\top A x + b^\top x + c$ with $A\in\R^{n\times n}$ symmetric. Let $C \colon [0,\infty) \to \R^{n\times n}$ solve the ODE $\dot{C}(t) = -2C(t)AC(t)$ of the AGRF under $f$ given by \eqref{eq: quadratic_flow}. Then $m\colon [0,\infty) \to \Rn$ given by
	\begin{equation}
		m(t) = C(t) \left(C(0)^{-1}m(0) - tb\right)
		\label{eq: quadratic_flow_mean}
	\end{equation}
	solves the ODE $\dot{m}(t) = -2C(t) A m(t) - C(t)b$ of the AGRF under $f$ given by \eqref{eq: quadratic_flow}.
\end{proposition}

\begin{proof}
	First, notice that the proposed solution $m$, given by \eqref{eq: quadratic_flow_mean}, satisfies the required initial condition:
	\begin{equation*}
		m(0) = C(0) \left(C(0)^{-1}m(0) - 0\cdot b\right) = m(0).
	\end{equation*}
	Next, for the proposed solution $m$, we find that
	\begin{align*}
		\dot{m}(t) &= \frac{d}{dt} \left(C(t) \left(C(0)^{-1}m(0) -tb\right)\right) \\
		&= \dot{C}(t) \left(C(0)^{-1}m(0)-tb\right) - C(t)b \\
		&= -2 C(t) A C(t) (C(0)^{-1}m(0) - tb) - C(t)b\\
		&= -2C(t)A m(t) - C(t)b,
	\end{align*}
	where we used the fact that $C$ solves $\dot{C}(t) = -2C(t)AC(t)$. Hence, $m$ indeed solves the desired ODE.
	% Note: The formula $\frac{d}{dt} \left(F(t)^{-1}\right) = -F(t)^{-1} \frac{dF}{dt}(t) F(t)^{-1}$ for a function $F \colon [0,\infty) \to \R^{n\times n}$ with invertible matrix values, is useful in related analyses and alternative proofs.
\end{proof}

We now discuss the limiting behavior of AGRFs under quadratic objectives. Consider the mean trajectory \eqref{eq: quadratic_flow_mean}. When $A$ is positive definite, we have that 
\begin{align*}
\lim_{t\to\infty}(C(0)^{-1}+2tA)^{-1} &= 0, \\ 
\lim_{t\to\infty} t(C(0)^{-1}+2tA)^{-1} &= \frac{1}{2} A^{-1},
\end{align*}
implying that $m$ stabilizes to
\begin{align*}
	\lim_{t\to\infty} m(t) &= \lim_{t\to\infty} \left(C(0)^{-1}+2tA\right)^{-1}\left(C(0)^{-1}m(0) - tb\right) \\
	&= -\frac{1}{2}A^{-1}b.
\end{align*}
This is precisely the solution to the optimization 
\begin{equation*}
	\inf_{x\in\Rn}f(x) = \inf_{x\in\Rn} \left(x^\top A x + b^\top x + c\right),
\end{equation*}
so indeed, AGRFs exactly solve convex quadratic optimization problems.

Notice that one may write the dynamics for $m$ as a preconditioned gradient flow:
\begin{align*}
	\dot{m}(t) &= -\left(C(0)^{-1}+2tA\right)^{-1}(2Am(t)+b) \\
	&= -\left(C(0)^{-1}+t\nabla^2 f(m(t))\right)^{-1}\nabla f(m(t)).
\end{align*}
For small $t$, the dynamics resemble a gradient flow that is weighted according to the initial covariance at small times, representing the initial uncertainty in where a minimum resides. As time evolves, the flow begins to more closely resemble Newton's method. \citet{jacimovic2022natural} showed that such replicator dynamics evolving under quadratic functions $f$ and restricted to the statistical manifold of Gaussian distributions follow a natural gradient flow with respect to the Fisher information metric. However, they did not explicitly determine the mean trajectory $m$, nor was the flow's preconditioning matrix identified as having the form we determined above.
% We note that these AGRF dynamics for $m$ are distinct from those induced by a standard Euclidean gradient flow:
% \begin{equation*}
% 	\dot{m}(t) = -\nabla f(m(t)) = -2Am(t)-b.
% \end{equation*}

\subsection{Descent of Flows}

In this section, we study when AGRFs descend the objective's landscape. We begin by showing that descent always occurs for AGRFs on a convex quadratic objective.

\begin{proposition}
	\label{prop: quadratic_descent_direction}
	Consider $f\colon \Rn \to \R$ given by $f(x) = x^\top A x + b^\top x + c$ with $A\in\R^{n\times n}$ symmetric. If $C(0)$ is positive definite, $A$ is positive semidefinite, and $(m,C)$ is the AGRF under $f$, then the velocity $\dot{m}(t)$ is a descent direction of $f$ for all $t\in [0,\infty)$, i.e.,
	\begin{equation*}
		\frac{d}{dt}f(m(t)) = \nabla f(m(t))^\top \dot{m}(t) < 0.
	\end{equation*}
\end{proposition}

\begin{proof}
	Since $\nabla f(x) = 2Ax + b$, we find from \autoref{prop: quadratic_flow_cov} that
	\begin{align*}
		&\nabla f(m(t))^\top \dot{m}(t) \\
		&\quad = -(2Am(t)+b)(C(0)^{-1}+2tA)^{-1}(2Am(t)+b) \\
		&\quad < 0
	\end{align*}
	whenever $m(t)$ is not a minimizer of $f$.
\end{proof}

\autoref{prop: quadratic_descent_direction} indicates that AGRF will converge to a local minimum of a general twice continuously differentiable function $f$ when the bulk of the Gaussian measure $\N(m(t),C(t))$ is concentrated in a neighborhood around the minimum that locally looks like a strictly convex quadratic. However, such descent need not always occur, as we now illustrate.

\begin{example}
	Consider $f \colon \R\to\R$ defined by $f(x) = \tfrac{3}{2}x^4 - \tfrac{1}{4}x^3 - 3x^2+\tfrac{3}{4}x + 1$. It is easy to verify that $f$ has one local minimum at $x=1$ with $f(1) = 0$, a global minimum at $x=-1$ with $f(-1)=-1$, and a local maximum in between. A simple computation gives that the AGRF mean satisfies
	\begin{align*}
		\dot{m}(t) &= -C(t)\left(6m(t)^3 - \frac{3}{4}m(t)^2 -6m(t) + \frac{3}{4}\right) \\
		&\qquad - C(t)^2\left(18m(t) - \frac{3}{4}\right) \\
		&= -C(t) f'(m(t)) - \frac{1}{2} C(t)^2 f'''(m(t)).
	\end{align*}
	The first term, $-C(t)f'(m(t))$, incentivizes the AGRF to perform local exploitation (descent along the gradient), whereas the second term, $-\tfrac{1}{2}C(t)^2 f'''(m(t))$, incentivizes the AGRF to explore for a better minimum. In particular, we may view the second term as descending along the gradient of $f''$; this term attempts to find a point where $f$ has the steepest negative curvature, i.e., a sharp peak. With a large enough covariance $C(t)$, this term dominates and as such results in a hill-climbing effect, so that the flow can effectively look for better minima.
\end{example}

The above example shows that AGRFs are capable of escaping local minima in order to reach a global minimum. We empirically verify such beneficial behavior in \autoref{sec: exper}.

We now discuss the descent behavior of AGRFs for more general functions $f$. For the remainder of this section, we will assume that $f$ has been adequately shifted so that $f(x) > 0$ for all $x\in X$, for convenience. The descent condition reads
\begin{align*}
	\frac{d}{dt} f(m(t)) &= \nabla f(m(t))^\top \dot{m}(t) \\
			     &= \E_{x\sim\N(m(t),C(t))}[f(x) \nabla f(m(t))^\top (m(t) - x)] \\
			     &<0.
\end{align*}
If the covariance $C(t)$ is small enough, then $x\sim \N(m(t),C(t))$ is close to $m(t)$ with high probability. For such points $x$ that are close to $m(t)$, the first-order Taylor estimate $\nabla f(m(t))^\top (m(t) - x) \approx f(m(t)) - f(x)$ serves as a good approximation. Therefore, the above descent condition is well-approximated by
\begin{equation}
	\E_{x\sim\N(m(t),C(t))}[f(x)(f(m(t)) - f(x))] < 0
	\label{eq: descent}
\end{equation}
for small $C(t)$. The approximate descent condition \eqref{eq: descent} requires that, on average (weighted by the function $f$), it holds that $f(m(t)) < f(x)$, according to our current distribution of strategies $\N(m(t),C(t))$, which one may view as reflecting our ``trust'' in where the optimizer is. From this perspective, it makes sense for the dynamics to descend the objective, since it is believed that, on average, the current mean is a good one, and hence we should exploit it and trust a local descent. On the contrary, if $f(m(t)) > f(x)$ on average, then the approximate descent condition is violated. This indicates a situation in which our current distribution $\N(m(t),C(t))$ believes that the mean $m(t)$ performs worse than the average of the other strategies, and therefore ascent occurs in order to explore for alternatives.

We conclude this section by remarking that (nonconstant) convex functions $f$ always satisfy the approximate descent condition \eqref{eq: descent}, since
\begin{align*}
&\E_{x\sim\N(m(t),C(t))}[f(x)]f(m(t)) - \E_{x\sim\N(m(t),C(t))}[f(x)^2] \\
&\quad \le \E_{x\sim\N(m(t),C(t))}[f(x)]^2 - \E_{x\sim\N(m(t),C(t))}[f(x)^2] \\
&\quad < 0
\end{align*}
by Jensen's inequality and the positivity of the variance of $f(x)$ under $x\sim\N(m(t),C(t))$.

% NOTE: Commenting out diagonal result due to problems in the bound and to save space
% \input{tex/diag.tex}

\section{Experiments}
\label{sec: exper}

\begin{figure*}[ht]
    \centering
    \subfloat[]{%
	\centering
        \includegraphics[width=0.27\textwidth]{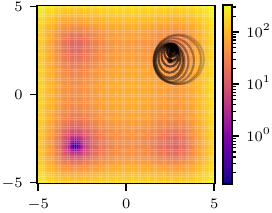}
        \label{fig:st_smallcov}
    }
    \hfill
    \subfloat[]{%
	\centering
        \includegraphics[width=0.27\textwidth]{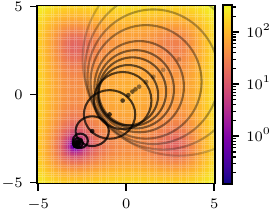}
        \label{fig:st_bigcov}
    }
    \hfill
    \subfloat[]{%
	\centering
        \includegraphics[width=0.27\textwidth]{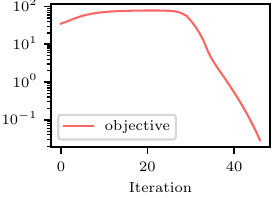}
        \label{fig:st_bigcov_objective}
    }
    \caption{First two images visualize mean (black dots) and covariance (circles) versus ODE time (light gray is first solver iteration, black is last solver iteration).
    \subref{fig:st_smallcov} Styblinski-Tang objective with a small initial covariance.
    \subref{fig:st_bigcov} Styblinski-Tang with a large initial covariance.
    \subref{fig:st_bigcov_objective} Objective value with respect to ODE iteration for trajectory in \subref{fig:st_bigcov}.}
    \label{fig:localmin}
\end{figure*}

This section illustrates our methods with proof-of-concept nonconvex optimization experiments. We use canonical test functions, commonly used for evaluating optimization algorithms, to showcase the qualitative features of AGRFs.

\inlinesubsection{Implementation}
We compute mean and covariance time derivatives using the formulae in \autoref{sec: simulating}.
%To ensure that the covariance matrix remains symmetric, we apply the transformation $\frac{1}{2} (\dot C + \dot C^{\top})$.
The resulting AGRF ODEs are integrated using an off-the-shelf Runge-Kutta 2-3 (RK23) solver. All optimizations are executed for a maximum of $T=30$ seconds of simulated time, terminating early when the covariance matrix determinant is smaller than $\epsilon = 0.0001$. We emphasize that, unlike most other optimization algorithms, our framework requires no hyperparameter choices beyond the initial conditions $m(0),C(0)$ and the simulation stopping time $T$, which are standard inputs required by off-the-shelf ODE solvers.

\subsection{Escaping Local Minima}
\label{sec: escape_minima}
\autoref{fig:localmin} illustrates the local minima-escaping behavior of our dynamics. We consider the Styblinski-Tang function, shifted to be more easily visualized on a $\log$ scale:
\begin{equation*}
    f(x) = 78.43 + \frac{1}{2} \left( \sum_{i=1}^n x_i^4 - 16 x_i^2 + 5x_i \right).
\end{equation*}
We let $n=2$, which gives rise to a global minimum at $(-2.903,-2.903)$ and three other local minima that are not global.

We fix an initial mean $m(0) = (3, 2)$ near the top-right local minimum. \autoref{fig:st_smallcov} initializes $C(0) = 2 I$, and \autoref{fig:st_bigcov} initializes $C(0) = 30 I$. Larger initial covariance intuitively helps the ODE ``see'' over the hump to eventually find the global minimum (\autoref{fig:st_bigcov_objective}). While the true replicator dynamics evolving over general probability measures would be guaranteed to find the global minimum by \autoref{prop: replicator_converges}, we observe in \autoref{fig:st_smallcov} that with a small initial covariance the ODE remains stuck in a local minimum. This arises from constraining our measures to be strictly Gaussian over the entire ODE trajectory.

Our framework naturally handles any Gaussian-integrable objective function, including the Styblinski-Tang polynomial objective. We show that sinusoids are also tractable by considering the classical Rastrigin function:
\begin{equation*}
    f(x) = 10 n + \sum_{i=1}^n \left(x_i^2 - 10 \cos(2 \pi x_i)\right).
\end{equation*}
As seen in \autoref{fig:rastrigin}, our algorithm replicates the local minima-escaping behavior observed above for a sufficiently large initial covariance of $C(0)=10I$ and initial mean of $m(0)=(4, 4)$.

\subsection{Optimization Landscape Adaptation}
\autoref{fig:quadratic} analyzes the tendency of the covariance dynamics to accelerate along directions of slow descent. Here, we consider the bivariate quadratic objective function given by
\begin{equation*}
    f(x) = (x_1-3)^2 + 4 (x_2 - 3)^2,
\end{equation*}
which exhibits a smaller gradient along the $x_1$-direction. The AGRF dynamics naturally stretch the covariance matrix along this direction, accelerating convergence (\autoref{fig:quadratic}). This is reminiscent of inverse Hessian preconditioning in Newton's method, and the effect of momentum in stochastic gradient descent-based optimization.

\begin{figure*}[ht]
    \centering
    \subfloat[]{%
	\centering
        \includegraphics[width=0.29\textwidth]{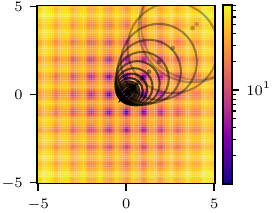}
        \label{fig:rastrigin}
    }
    \hspace*{1.5cm}
    \subfloat[]{%
	\centering
        \includegraphics[width=0.29\textwidth]{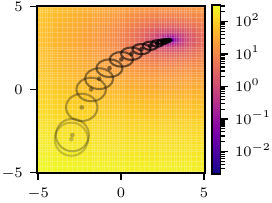}
        \label{fig:quadratic}
    }
    \\
    \subfloat[]{%
	\centering
        \includegraphics[width=0.29\textwidth]{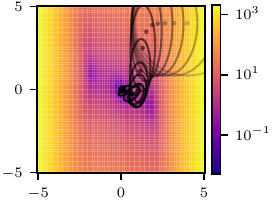}
        \label{fig:threehump_bigmean}
    }
    \hspace*{1.5cm}
    \subfloat[]{%
	\centering
        \includegraphics[width=0.29\textwidth]{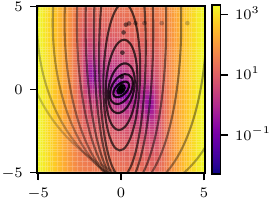}
        \label{fig:threehump_bigcov}
    }
    \caption{Plots visualize mean (black dots) and covariance (circles) versus ODE time (light gray is first solver iteration, black is last solver iteration).
    \subref{fig:rastrigin} Rastrigin objective.
    \subref{fig:quadratic} Quadratic objective $f(x)=(x_1-3)^2 + 4 (x_2-3)^2$.
    \subref{fig:threehump_bigmean} Three-hump objective with $m(0)=(4, 4)$ and $C(0)=10 I$.
    \subref{fig:threehump_bigcov} Three-hump objective with $m(0)=(4, 4)$ and $C(0)=100 I$. }
    \label{fig:landscape_and_init}
\end{figure*}

\subsection{Recovery from Poor Initialization}
\autoref{fig:threehump_bigmean} and \autoref{fig:threehump_bigcov} illustrate the robustness of our approach to poorly conditioned initial points. The objective is a classic three-hump camel function:
\[
    f(x) = 2 x_1^2 - 1.05 x_1^4 + \frac{x_1^6}{6} + x_1 x_2 + x_2^2.
\]
The unique global minimum of this function lies at the origin, with a local minimum on each side. At points with sufficiently large $x_1$-values, this function exhibits an ill-conditioned Hessian and a large-magnitude gradient that is nearly orthogonal to the $x_2$-direction. We test such a poorly conditioned initial mean of $m(0)=(4, 4)$, with initial covariance matrices of $C(0)=10 I$ for \autoref{fig:threehump_bigmean} and $C(0)=100 I$ for \autoref{fig:threehump_bigcov}. In both cases, the ODE succeeds in recovering from the poor initialization and converges to the global minimum. This suggests that our approach successfully leverages the adaptive step size capabilities of the underlying RK23 ODE solver.

% \subsection{Dimensionality scaling and runtime}

\section{Conclusion}
\label{sec: conc}

This work proposes a framework for recasting nonconvex optimization as an evolutionary game. Since simulating the resulting infinite-dimensional replicator dynamics is intractable, we propose to approximate the evolution by a trajectory of Gaussian distributions, termed an approximately Gaussian replicator flow. We theoretically characterize the flow's descent behavior and ability to ascend from local minima, in addition to fully determining the flow on quadratic objectives. Our proof-of-concept experiments on polynomial and sinusoidal nonconvex benchmark objectives demonstrate the potential of our method for finding global optima using off-the-shelf ODE solvers, without the need for heuristic algorithm design choices. Our results demonstrate that even a Gaussian approximation to the replicator dynamics exhibits a range of desirable phenomenon, including the ability to escape local minima, adapt to the optimization landscape, and recover from poor initialization.

% appendix

% \section*{Acknowledgments}
% Write acknowledgments here.

\footnotesize
\bibliographystyle{abbrvnat}
\bibliography{tex/references.bib}

% \newpage
% \onecolumn
% \input{tex/todo.tex}

\end{document}